\documentclass[]{interact}

\usepackage{epstopdf}
\usepackage[caption=false]{subfig}
\usepackage{geometry}
\usepackage{amsmath}
\usepackage{enumerate}
\usepackage{amsmath,amssymb}
\usepackage{amsthm}
\usepackage{enumitem}
\usepackage{cite}
\usepackage{fancyhdr}
\usepackage[numbers,sort&compress]{natbib}
\bibpunct[, ]{[}{]}{,}{n}{,}{,}
\renewcommand\bibfont{\fontsize{10}{12}\selectfont}
\makeatletter
\def\NAT@def@citea{\def\@citea{\NAT@separator}}
\makeatother

\usepackage{color,xcolor}
\theoremstyle{plain}
\newtheorem{theorem}{Theorem}[section]
\newtheorem{corollary}[theorem]{Corollary}
\newtheorem{lemma}[theorem]{Lemma}
\newtheorem{remark}[theorem]{Remark}
\newtheorem{example}[theorem]{Example}

\theoremstyle{plain}
\newtheorem*{remark*}{Remark}
\theoremstyle{plain}
\newtheorem*{corollary*}{Corollary}
\theoremstyle{plain}
\newtheorem*{example*}{Example}
\theoremstyle{definition}
\newtheorem*{Definition*}{Definition}
\numberwithin{equation}{section}

\title{{Orthogonal tripotent matrices}}
\author{\name{Tan Mei\textsuperscript{a}, 
Kezheng Zuo\textsuperscript{a,$\ast$}\thanks{$^{\ast}$ Corresponding author: xiangzuo28@163.com (Kezheng Zuo)} 
and Wanlin Jiang\textsuperscript{b}}
\affil{\textsuperscript{a} Department of Mathematics, School of Mathematics and Statistics, Hubei Normal University, Huangshi 435002, China;}
{\textsuperscript{b} Department of Mathematics, School of Science and Technology, College of Arts and Science of Hubei Normal University, Huangshi 435002, China. }
}
\begin{document}

\maketitle

  \begin{abstract}
In this paper, we present different characterizations of tripotent orthogonal matrices (i.e., $A^3 = A = A^* $) in terms of  matrix equations, integer powers of $AA^*$ and $A^*A$, average of $A$, $A^*$, and $A^{\dagger}$, rank of matrices, and trace of matrices. We study certain properties of this class of matrices.
    \end{abstract}

    \begin{keywords}
      Orthogonal tripotent matrix; Orthogonal $k$-idempotent matrix; Hartwig-Spindelb\"{O}ck decomposition; Moore-Penrose inverse.
      \end{keywords}

  {\small \textbf{2020 MATHEMATICS SUBJECT CLASSIFICATION}   {15A09 } } 

\section{Introduction and preliminaries} 

Orthogonal idempotent matrices (i.e.,  $A^2 = A = A^*$, where $A^*$ denotes the conjugate transpose of $A$) is an important class with wide applications in various fields, including least squares methods in statistics, control theory, operator theory, and matrix equations (see \cite{Toutenburg1991, Riesz2012, Fisher1991,Deng,Tosic,Zou}).

\noindent More generally, we call $A$ an orthogonal $k$-idempotent matrix if it satisfies $A^k = A = A^*$ for some integer $k \geq 2$. The Hermitian condition $A = A^*$ ensures that all eigenvalues of $A$ are real numbers, so for Hermitian $A$ the condition $A^k = A$ restricts them to $\pm 1$ or $0$. Consequently, any such matrix $A$ is unitarily similar to a diagonal matrix of the form $\mathrm{diag}(1, \ldots, 1, -1, \ldots, -1, 0, \ldots, 0)$. This diagonalization immediately implies that $A^3 = A = A^*$ holds for all orthogonal $k$-matrices. 
For this reason, the study of the class of all orthogonal tripotent matrices holds particular significance, and it will constitute the main subject of this paper. \\
\noindent  Standardly, the set of $m \times n$ complex matrices is denoted by $\mathbb{C}^{m \times n}$. 
The symbols $\mathcal{R}(A)$ and $\operatorname{r}(A)$ stand for range and rank of $A \in \mathbb{C}^{m \times n}$, whereas for $A \in \mathbb{C}^{n \times n}$, $\operatorname{tr}(A)$ denote the trace. 
And let $\mathbb{Z}$, $\mathbb{Z}^+$ and $\mathbb{C}$ denote the sets of integers, positive integers and complex numbers, respectively. 
Throughout this paper, we adopt the following notation 
\begin{eqnarray}\label{1.1}
\mathbb{C}_n^{3\text{-}OP} = \{ A \mid A \in \mathbb{C}^{n \times n}, A^3 = A = A^* \}.
\end{eqnarray}
Furthermore, $A^{\dagger} \in \mathbb{C}^{m \times n}$ 
is the Moore-Penrose inverse of $A \in \mathbb{C}^{m \times n}$, i.e.,  the unique solution to the equations(see \cite{Ben2003,Penrose,Cvetkovi,Gao2024}):
\begin{eqnarray*}
AA^{\dagger}A = A, \quad A^{\dagger}AA^{\dagger} = A^{\dagger}, \quad (AA^{\dagger})^{*} = AA^{\dagger}, \quad (A^{\dagger}A)^{*} = A^{\dagger}A.
\end{eqnarray*}
\noindent Finally, let $I_n$ be the $n \times n$ identity matrix, and $\mathbb{C}_n^U$ be the set of $n \times n$ unitary matrices. For $a \in \mathbb{C}$, $\operatorname{Re}(a)$ denotes its real part, and $i=\overline{1,r}$ stands for $i=1,2,\cdots,r$.

\noindent  The main purpose of this paper is to study the properties and characterizations of orthogonal tripotent matrices. Our main contributions can be summarized as follows:
\begin{itemize}
\item We derive fundamental properties of orthogonal tripotent matrices.
\item Using the Hartwig-Spindelb\"{o}ck decomposition of matrices and analyzing the mean equations of $A$, $A^{*}$, and $A^{\dagger}$, we characterize orthogonal tripotent matrices. 
Furthermore, we establish necessary and sufficient conditions for $A \in \mathbb{C}_{n}^{3\text{-}OP}$ through combined equations involving $AA^{*}$ and $A^{*}A$ with integer exponents.
\item Using matrix properties including tripotent, Hermitian and EP matrices, we  characterize orthogonal tripotent matrices.  
\end{itemize}

\noindent The paper is organized as follows: In Section 1, necessary lemmas and essential matrix classes are introduced. We devote Section 2 to establishing fundamental properties of orthogonal tripotent matrices. And several characterizations of orthogonal tripotent matrices are presented. Concluding remarks are given in Section 3.

\noindent To investigate orthogonal tripotent matrices, we first state the following lemmas.

    
\begin{lemma}\label{y1} {\rm\cite{Gao2024, Baksalary2022, Hartwig1983}}\
    Let $A\in\mathbb{C}^{n\times n}$ with $\operatorname{r}(A) = r$. Then there exists $U\in\mathbb{C}_{n}^{U}$ such that 
    \begin{equation}\label{2.1}
        A = U\begin{bmatrix}
            \Sigma K & \Sigma L\\
            0 & 0
        \end{bmatrix}U^{*},
    \end{equation}
    where $\Sigma=\operatorname{diag}(\sigma_1,\sigma_2,\ldots,\sigma_r)$ is the diagonal matrix of singular values of $A$, $\sigma_i > 0$ for $i = \overline{1,r}$, $K\in\mathbb{C}^{r\times r}$, $L\in\mathbb{C}^{r\times(n - r)}$, and
    \begin{equation}\label{2.2}
        KK^{*}+LL^{*}=I_r.
    \end{equation}
\end{lemma}
   
\noindent The lemma above plays a crucial role in analyzing special matrix properties and in characterizing new classes of generalized inverses (see \cite{Baksalary2022, Khatri1980, Malik2025, Ferreyra, Sitha2023}).

\noindent  Using the above  decomposition of $A$ well-known as Hartwig-Spindelb\"{o}ck decomposition, $A^{\dagger}$ can be expressed as (see \cite{Maria2008}):
    \begin{eqnarray}\label{eMP}
    A^{\dagger} = U\begin{pmatrix}
    K^*\Sigma^{-1} & 0 \\
    L^*\Sigma^{-1} & 0
    \end{pmatrix}U^*.
    \end{eqnarray}

\begin{lemma}\label{y2}
Let $K \in \mathbb{C}^{r \times r}$ and $\Sigma = \text{diag}(\sigma_1, \sigma_2, \ldots, \sigma_r)$ be a diagonal matrix with $\sigma_i > 0$ for $i = \overline{1,r}$. Then $\Sigma K = K \Sigma$ if and only if one of the following conditions holds:
\begin{itemize}
    \item [$(a)$] $\Sigma^s K = K \Sigma^s$ for $s \in \mathbb{Z}\setminus\{0\}$;
    \item  [$(b)$]$(\Sigma^s + \Sigma^t) K = K (\Sigma^s + \Sigma^t)$ for $s,t \in \mathbb{Z}$ with $st \geq 0$ and $(s,t) \neq (0,0)$.
\end{itemize}
\end{lemma}

\begin{proof}
The necessity is clear; we now prove sufficiency.  
Let $K = (k_{ij})_{r \times r}$. To show that $\Sigma K = K\Sigma$, it suffices to verify
\begin{equation}\label{2.4}
\sigma_i k_{ij} = k_{ij} \sigma_j \quad (i,j = \overline{1,r}). 
\end{equation}

\noindent $(a):$ Since $\Sigma^s K = K\Sigma^s$ holds if and only if $\Sigma^{-s} K = K\Sigma^{-s}$, we may assume $s > 0$.  
If $k_{ij} \neq 0$, then
\[
\sigma_i^s k_{ij} = k_{ij}\sigma_j^s \quad \Rightarrow \quad \sigma_i^s = \sigma_j^s.
\]
Since  $\sigma_i, \sigma_j > 0$, we obtain $\sigma_i = \sigma_j$, and thus (\ref{2.4}) holds. The case when $k_{ij}=0$ is the trivial one. Hence, $\Sigma K = K\Sigma$.

\noindent $(b):$ If $k_{ij} \neq 0$, then
\begin{equation}\label{2.5}
(\sigma_i^s+\sigma_i^t)k_{ij} = k_{ij}(\sigma_j^s+\sigma_j^t)
\quad \Rightarrow \quad \sigma_i^s - \sigma_j^s + \sigma_i^t - \sigma_j^t = 0. 
\end{equation}

\noindent For $s,t>0$, (\ref{2.5}) gives
\[
(\sigma_i - \sigma_j)\left(\sum_{p=0}^{s-1}\sigma_i^p\sigma_j^{s-1-p} + \sum_{p=0}^{t-1}\sigma_i^p\sigma_j^{t-1-p}\right)=0,
\]
and by the positive of $\sigma_i,\sigma_j$ we get  $\sigma_i = \sigma_j$, ensuring (\ref{2.4}).  

\noindent For $s,t< 0$, let $s=-m$, $t=-n$ with $m,n\in\mathbb{Z}^+$. Then (\ref{2.5}) becomes
\[
(\sigma_j-\sigma_i)\left(\sum_{p=0}^{m-1}\sigma_i^{p-m}\sigma_j^{-(p+1)} + \sum_{p=0}^{n-1}\sigma_i^{p-n}\sigma_j^{-(p+1)}\right)=0,
\]
and again by the positive of $\sigma_i, \sigma_j$ we get  $\sigma_i=\sigma_j$, yielding (\ref{2.4}).  The case when $k_{ij}=0$ is the trivial one. The case when $t=0$ or $s=0$ reduces to the case $(a)$. 
\noindent  Therefore, (\ref{2.4}) holds for all $s,t\in\mathbb{Z}$ with $st\geq 0$ and $(s,t)\neq(0,0)$. 
\end{proof}

\noindent The following general result holds: 

\begin{remark}\label{rrr}
     Let $ A \in \mathbb{C}^{r \times r} $, and let $\Sigma = \operatorname{diag}(\sigma_1, \sigma_2, \ldots, \sigma_r)$ be a diagonal matrix with $\sigma_i > 0$ for $i = \overline{1,r}$. 
     Then, for non-negative integers $s_1, s_2, \ldots, s_t$ not all zero,
$\left( \Sigma^{s_1} + \Sigma^{s_2} + \cdots + \Sigma^{s_t} \right) A = A \left( \Sigma^{s_1} + \Sigma^{s_2} + \cdots + \Sigma^{s_t} \right)$ holds if and only if $\Sigma A = A \Sigma$.
\end{remark}

\begin{remark}
\noindent The converse implication doesn't hold and that will be illustrated by the following example. Namely, we will show that in general the condition $(\Sigma^s + \Sigma^t)A = A(\Sigma^s + \Sigma^t)$ when $st<0$  does not necessarily imply $\Sigma A = A\Sigma$.
\end{remark}
\begin{example}\label{range}
Let $\Sigma = \operatorname{diag}\!\left(2, \tfrac{1}{2}\right)$ and let  $s = 3$, $t = -3$. Then
\[
\Sigma^3 + \Sigma^{-3} 
= \operatorname{diag}\!\left(2^3 + 2^{-3}, \left(\tfrac{1}{2}\right)^3 + \left(\tfrac{1}{2}\right)^{-3}\right) 
= \operatorname{diag}\!\left(\tfrac{65}{8}, \tfrac{65}{8}\right).
\]
Let $A=\begin{pmatrix}
0 & 1 \\
1 & 0
\end{pmatrix}. $ A direct computation shows
\[
(\Sigma^3 + \Sigma^{-3})A = A(\Sigma^3 + \Sigma^{-3}) 
= 
\begin{pmatrix}
0 & \tfrac{65}{8} \\
\tfrac{65}{8} & 0
\end{pmatrix}.
\]
However, $\Sigma A =
\begin{pmatrix}
0 & 2 \\
\tfrac{1}{2} & 0
\end{pmatrix} \neq  A \Sigma =
\begin{pmatrix}
0 & \tfrac{1}{2} \\
2 & 0
\end{pmatrix}. $
\end{example}

\noindent To facilitate subsequent discussion, we introduce several classes of special matrices. Denote by $\mathbb{C}_n^{H}$, $\mathbb{C}_n^{OP}$, $\mathbb{C}_n^{TM}$, $\mathbb{C}_n^N$, $\mathbb{C}_n^{EP}$, $\mathbb{C}_n^{MP}$, $\mathbb{C}_n^{SD}$, and $\mathbb{C}_n^{PI}$ the sets of Hermitian, orthogonal projection, tripotent, normal, EP, Moore-Penrose invertible, star-dagger, and partial isometry matrices, respectively, that is,
\begin{align*}
& \mathbb{C}_n^{H} = \{A\in\mathbb{C}^{n\times n}:A^* = A\},\\
&\mathbb{C}_n^{P} = \{A\in\mathbb{C}^{n\times n}:A^2 = A\},\\
&\mathbb{C}_n^{OP} = \{A\in\mathbb{C}^{n\times n}:A^2 = A^* = A\},\\
&\mathbb{C}_n^{TM}= \{A\in\mathbb{C}^{n\times n}:A^3 = A\}, \\
&\mathbb{C}_n^N = \{A \in \mathbb{C}^{n\times n} : AA^* = A^*A\},\\
&\mathbb{C}_n^{EP} = \{A \in \mathbb{C}^{n\times n} : AA^{\dagger} = A^{\dagger}A\} = \{A \in \mathbb{C}^{n\times n} :\mathcal{R} (A) = \mathcal{R} (A^*)\},\\
&\mathbb{C}_n^{MP}=\{A\in\mathbb{C}^{n\times n}:A^{\dagger} = A\},\\
&\mathbb{C}_n^{SD} = \{A \in \mathbb{C}^{n\times n} : A^{\dagger}A^* = A^*A^{\dagger}\},\\
&\mathbb{C}_n^{PI}=\{A\in\mathbb{C}^{n\times n}:A = AA^*A\} = \{A\in\mathbb{C}^{n\times n}:A^* = A^{\dagger}\}.
\end{align*}

\begin{lemma}\label{y5}
    Let $A$ be given by \textup{(\ref{2.1})}. The following statements hold:
  
\noindent  $(a)$ $A \in \mathbb{C}_n^{H} \Leftrightarrow L = 0, (\Sigma K)^* = \Sigma K$;\quad 
 $(b)$ $A \in \mathbb{C}_n^{MP} \Leftrightarrow L = 0, (\Sigma K)^2 = I_r$;\\
$(c)$ $A \in \mathbb{C}_n^{N} \Leftrightarrow L = 0,  K\Sigma = \Sigma K$;\quad 
$(d)$ $A \in \mathbb{C}_n^{TM} \Leftrightarrow (\Sigma K)^2 = I_r$;\\
$(e)$  $A \in \mathbb{C}_n^{EP} \Leftrightarrow L = 0$;\quad 
$(f)$  $A \in \mathbb{C}_n^{PI} \Leftrightarrow \Sigma = I_r$;\quad 
$(g)$  $A \in \mathbb{C}_n^{SD} \Leftrightarrow K\Sigma = \Sigma K$;\\
$(h)$ $A \in \mathbb{C}_{n}^{3\text{-}OP} \Leftrightarrow L = 0,\ (\Sigma K)^2 = I_r,\ (\Sigma K)^* = \Sigma K
    \Leftrightarrow L = 0,\ \Sigma = I_r,\ K^2= I_r
    \Leftrightarrow L = 0,\ \Sigma= I_r,\ K = K^*.$
 \end{lemma}    
\begin{proof} We prove conditions ($b$) and ($h$) only; condition ($a$) is immediate, and the remaining equivalences follow from \cite{Baksalary,Baksalary2,Zuo}.

\noindent ($b$)  Let  $A \in \mathbb{C}_n^{MP}$ be given by (\ref{2.1}). Then $A^2 \in \mathbb{C}_n^{H}$, which gives that $K\Sigma L=0$. And $A^3=I_r$ which implies that 
 $(\Sigma K)^3=\Sigma K$ and $(\Sigma K)^2\Sigma L=\Sigma L$. Evidently, $\Sigma L=\Sigma K\Sigma (K \Sigma L)=0$, i.e.,  $L=0$. This implies that $KK^*=I_r$ which together with $(\Sigma K)^3=\Sigma K$ gives that $(\Sigma K)^2=I_r$. Conversly, if we assume that $L = 0$ and  $(\Sigma K)^2 = I_r$, it is easy to verify that $A \in \mathbb{C}_n^{MP}$.

\noindent ($h$)   By (\ref{1.1}), $\mathbb{C}_n^{3\text{-}OP} = \mathbb{C}_n^{TM} \cap \mathbb{C}_n^{H}$.  
From ($a$) and ($d$), it follows that $A$ satisfies $A^3 = A = A^*$ if and only if
\[
L = 0, \quad (\Sigma K)^2 = I_r \quad \text{and} \quad K^*\Sigma = \Sigma K.
\]
Moreover, $(\Sigma K)^2 = \Sigma KK^*\Sigma = I_r$. Since $L = 0$ and (\ref{2.2}) holds, we obtain $\Sigma^2 = I_r$. As $\Sigma$ is diagonal with positive entries, it follows that $\Sigma = I_r$, and hence $K^2 = I_r$. Finally, $K$ satisfies $K = K^*$.
\end{proof}

\section{On the properties of orthogonal tripotent matrices}

The following elementary result present a fundamental structural property of orthogonal tripotent matrices.  

\begin{theorem}\label{zq} If $A$ is an orthogonal tripotent matrix, then there exists $U \in \mathbb{C}_n^{U}$ such that  
\begin{equation}\label{e111}
A = U \, \mathrm{diag}(\underbrace{1, \ldots, 1}_{p}, \underbrace{-1, \ldots, -1}_{q}, \underbrace{0, \ldots, 0}_{r}) U^*,
\end{equation}
for some nonnegative integers $p, q, r$ such that $p+q+r = n$.  
Conversely, any matrix of this form is an orthogonal tripotent matrix.  
\end{theorem}

\begin{proof}  Let $A \in \mathbb{C}_{n}^{3\text{-}OP}$. Since $A\in \mathbb{C}_n^{H}$, all eigenvalues of $A$ are real and $A$ is diagonalizable.  Moreover, because  of $A^3 = A$, for any eigenvalues $\lambda$ of $A$, we have that $\lambda^3=\lambda$, i.e., $\lambda \in \{0, 1, -1\}$. Thus, there exists $U \in \mathbb{C}_n^{U}$ such that $A$ is given by $(\ref{e111})$ for some nonnegative integers $p, q, r$ with $p+q+r = n$.  

\noindent Conversely, any matrix of the above form clearly satisfies $A^* = A$ and $A^3 = A$, and hence belongs to $\mathbb{C}_n^{3\text{-}OP}$ by (\ref{1.1}).
\end{proof}

\begin{remark}
\noindent Similarly as in Theorem~\textup{\ref{zq}}, we can extend the characterization of orthogonal $k$-matrix for natural number $k \geq 2$:\\
  $(a)$ when $k$ is even, 
    \[
        A^k = A = A^* \Leftrightarrow A^2 = A = A^*;
    \]
     $(b)$  when $k$ is odd,
    \[
        A^k = A = A^* \Leftrightarrow A^3 = A = A^*.
    \]
\end{remark}


\begin{theorem}\label{t3}
A matrix $A \in \mathbb{C}^{n \times n}$ belongs to $\mathbb{C}_{n}^{3\text{-}OP}$ if and only if 
\[
A = A^* = A^\dagger.
\]
\end{theorem}

\begin{proof}
$(\Rightarrow)$ Let $A \in \mathbb{C}_{n}^{3\text{-}OP}$. Then $A = A^*=A^3$.  
By Theorem~\ref{zq}, $A$ is given by $(\ref{e111})$, so evidently  $A = A^* = A^\dagger$. 

\noindent $(\Leftarrow)$ Conversely, assume that $A = A^* = A^\dagger$.  Then $A=AA^\dagger A=AAA=A^3$, so $A \in \mathbb{C}_{n}^{3\text{-}OP}$.
\end{proof}


\noindent  The next theorem characterizes the relationship between orthogonal tripotent matrices and various fundamental classes of matrices.

\begin{theorem}\label{ze}
    The following  hold:

\noindent $(a)$ $\mathbb{C}_n^{3\text{-}OP} = \mathbb{C}_n^{3\text{-}P} \cap \mathbb{C}_n^{N}$; \quad 
      $(b)$ $\mathbb{C}_n^{3\text{-}OP} = \mathbb{C}_n^{H} \cap \mathbb{C}_n^{PI}$;\quad  $(c)$ $\mathbb{C}_n^{3\text{-}OP} = \mathbb{C}_n^{H} \cap \mathbb{C}_n^{MP}$;\\
      $(d)$ $\mathbb{C}_n^{3\text{-}OP} = \mathbb{C}_n^{TM} \cap \mathbb{C}_n^{EP} \cap \mathbb{C}_n^{PI}$;   \quad 
  $(e)$ $\mathbb{C}_n^{3\text{-}OP} = \mathbb{C}_n^{TM} \cap \mathbb{C}_n^{EP} \cap \mathbb{C}_n^{SD}$; \\ 
$(f)$ $\mathbb{C}_n^{3\text{-}OP} = \mathbb{C}_n^{MP} \cap \mathbb{C}_n^{SD}$;   \quad 
$(g)$ $\mathbb{C}_n^{3\text{-}OP} = \mathbb{C}_n^{MP} \cap \mathbb{C}_n^{PI}$.

\end{theorem}
\begin{proof}
  Noting that $\mathbb{C}_n^{N} = \mathbb{C}_n^{EP} \cap \mathbb{C}_n^{SD}$ and 
$\mathbb{C}_n^{MP} = \mathbb{C}_n^{TM} \cap \mathbb{C}_n^{EP}$ by Lemma \ref{y5}, we see that $(e)$ follows directly from $(a)$, and $(f)$--$(g)$ follow directly from $(d)$--$(e)$. 
Hence, it suffices to establish $(a)$--$(d)$.\\
\noindent $(a):$ The inclusion $\mathbb{C}_n^{3\text{-}OP} \subset \mathbb{C}_n^{TM} \cap \mathbb{C}_n^N$ follows immediately from Lemma~\ref{y5}. 
    Conversely, suppose $A \in \mathbb{C}_n^{TM} \cap \mathbb{C}_n^N$. By Lemma \ref{y5}, we obtain
    \begin{eqnarray*}
        (\Sigma K)^2 = I_r, \quad L = 0 \quad \text{and} \quad K\Sigma = \Sigma K.
    \end{eqnarray*}
    From the first and third relations, we deduce 
    \[
        (\Sigma K)^2 = \Sigma KK \Sigma = I_r,
    \]
    which implies $K^2 = \Sigma^{-2}$. Combining this with $L = 0$ yields 
    \begin{eqnarray*}
    \Sigma^{-4} = K^2 (K^*)^2 = I_r.
    \end{eqnarray*}  
   Consequently, we obtain $\Sigma = I_r$, which upon substitution into $(\Sigma K)^2 = I_r$ yields $K^2 = I_r$. 
By Lemma~\ref{y5}, it follows that $A \in \mathbb{C}_n^{3\text{-}OP}$. \\
\noindent $(b):$ This assertion follows directly from Lemma \ref{y5}.\\
\noindent $(c):$ It is clear that $\mathbb{C}_n^{3\text{-}OP} \subset \mathbb{C}_n^{H} \cap \mathbb{C}_n^{MP}$. 
Conversely, let $A \in \mathbb{C}_n^{H} \cap \mathbb{C}_n^{MP}$. 
By Lemma~\ref{y5}, we obtain
\[
L = 0, \quad (\Sigma K)^* = \Sigma K \quad \text{and} \quad (\Sigma K)^2 = I_r,
\]
which together imply $(\Sigma K)^2 = \Sigma K K^* \Sigma = I_r$. 
Combining this with $L = 0$ and (\ref{2.2}), a direct computation shows that $\Sigma = I_r$. 
Thus, by Lemma \ref{y5}, we have $A \in \mathbb{C}_n^{3\text{-}OP}$.\\
\noindent $(d):$ Clearly, $\mathbb{C}_n^{3\text{-}OP} \subset \mathbb{C}_n^{TM} \cap \mathbb{C}_n^{EP} \cap \mathbb{C}_n^{PI}$. 
Conversely, let $A \in \mathbb{C}_n^{TM} \cap \mathbb{C}_n^{EP} \cap \mathbb{C}_n^{PI}$. 
By Lemma \ref{y5}, we have
\[
(\Sigma K)^2 = I_r, \quad L = 0 \quad \text{and} \quad \Sigma = I_r.
\]
These relations imply $K^2 = I_r$. Then Lemma \ref{y5} yields $A \in \mathbb{C}_n^{3\text{-}OP}$.
\end{proof}
  


\begin{theorem}\label{zr}
   $A$ is an orthogonal tripotent matrix if and only if it satisfies one of the following equivalent conditions:\\
     $(a)$  $A = E - F$, where $E, F \in \mathbb{C}_n^{OP}$ and $EF = FE$;\\
     $(b)$  $A = H + L$, where $H$ and $L$ are in $\mathbb{C}_n^{3 - OP}$ and $HL = LH = 0$.
\end{theorem}
\begin{proof}  The necessity in both items follows directly from the identity $A = A - 0 = A + 0$,  while the sufficiency follows directly by computation.  
\end{proof}


\noindent  In the next theorem, we derive a characterization of an orthogonal tripotent matrix by means of singular value decomposition.

\begin{theorem}\label{zt} Let $ A \in \mathbb{C}^{n \times n}$ and $ r=\operatorname{r}(A) $. Then $ A \in \mathbb{C}_{n}^{3\text{-}OP} $ if and only if there exist $ U_1, V_1 \in \mathbb{C}^{n \times r} $ such that  
$$
 U_1^*U_1 = V_1^*V_1 = I_r, \quad A = U_1V_1^* = V_1U_1^* \quad \text{and} \quad (U_1^*V_1)^2 = I_r.
$$
\end{theorem}
\begin{proof} 
\noindent  \textbf{($\Leftarrow$)} The result follows by direct verification.\\

\noindent \textbf{($\Rightarrow$)}
Let the singular value decomposition of $ A $ be
$$
A = U \begin{pmatrix}
\Sigma & 0 \\
0 & 0
\end{pmatrix} V^{*},
$$
where $ \Sigma = \operatorname{diag}(\sigma_{1},\ldots,\sigma_{r}) $ with $ \sigma_{i} > 0 $, $ \sigma_{i} \in \mathbb{R} $, $i=\overline{1,r}$ and $ U, V \in \mathbb{C}_{n}^{U} $.  
Let $$U = 
\begin{bmatrix}
U_1 & U_2
\end{bmatrix} \text{ and } 
V =
\begin{bmatrix}
V_1^* \\
V_2^*
\end{bmatrix},$$
where $U_1, V_1 \in \mathbb{C}^{n \times r}$. By $U^*U=VV^*=I_n$, we have  $U_1^*U_1 = V_1^*V_1 = I_r $. Also  $A  =U_1\Sigma V_1^*$. Since $A^* = A $, we have $V_1 \Sigma U_1^* = U_1 \Sigma V_1^* $. From $A^3 = A $, it follows that
\begin{eqnarray*} 
U_1\Sigma V_1^*V_1 \Sigma U_1^*U_1\Sigma V_1^* = U_1\Sigma V_1^* ,
\end{eqnarray*} 
i.e., $U_1\Sigma^3 V_1^* = U_1\Sigma V_1^* $. Therefore, $\Sigma^3 = \Sigma $ which implies that $\Sigma = I_r $. Thus  $A = U_1V_1^* = V_1U_1^* $ and $(U_1V_1^*)^3=U_1V_1^*$, i.e., $(U_1^*V_1)^2 = I_r $.
\end{proof}

\noindent Building on Theorem~\ref{t3}, we consider that orthogonal tripotent matrices can be characterized through averages of $A$, $A^*$, and $A^\dagger$.

\begin{theorem}\label{za}  Let $ A \in \mathbb{C}^{n \times n}$. The following are equivalent: \\
    $(a)$   $A$ is an orthogonal tripotent matrix;\\
     $(b)$  $\frac{1}{3}(A + A^{*}+A^{\dagger}) = A$;\\
     $(c)$  $\frac{1}{3}(A + A^{*}+A^{\dagger}) = A^{\dagger}$.
\end{theorem}
\begin{proof}  We have that $(a)\Rightarrow (b)$ and  $(a)\Rightarrow (c)$   follow by Theorem \ref{t3}. 

\noindent $(b)\Rightarrow(a): $     We have that  $A^* + A^{\dagger}=2A$. If we assume that $A$ is given by $(\ref{2.1})$, then by $(\ref{eMP})$ we have that 
        \begin{equation}\label{4.1}
            L = 0 \quad \text{and} \quad K^*(\Sigma+\Sigma^{-1}) = 2\Sigma K .
        \end{equation}
        Thus, $2\Sigma K K^*(\Sigma+\Sigma^{-1}) = 2K^*(\Sigma+\Sigma^{-1}) \Sigma K $, i.e.,  $
            K(\Sigma^2 + I_r)=(\Sigma^2 + I_r)K.$ 
        By Lemma~\ref{y2}, we have $K\Sigma=\Sigma K$, so by (\ref{4.1}), we get  
        \begin{equation}\label{4.2}
            K^2=\frac{1}{2}(I_r+\Sigma^{-2}).
        \end{equation}
        Since $K\in \mathbb{C}_r^U$ ($KK^*= I_r$), we have that $K$ can be diagonalized as
        \begin{eqnarray*}
            K = V \text{diag}(k_1,k_2,\cdots,k_r) V^*,
        \end{eqnarray*}
        where $V \in \mathbb{C}_r^U$ and $k_j\in\mathbb{C}$ for $j = \overline{1,r}$ and $|k_j|^2 = 1$. Together with (\ref{4.2}), this gives
        $$
            \frac{1}{2}|1+\sigma_j^{-2}| = |k_j|^2 = 1,
        $$
        which implies $\sigma_j = 1$ for all $j$ and consequently $\Sigma = I_r$. Substituting $\Sigma = I_r$ into (\ref{4.2}) we have that  $K^2 = I_r$. Hence by Lemma~\ref{y5}  it follows that $A \in \mathbb{C}_n^{3\text{-}OP}$ .
        
\noindent $(c)\Rightarrow(a):$   We have that $A + A^{*}=2A^{\dagger}$. If we assume that $A$ is given by $(\ref{2.1})$, then by $(\ref{eMP})$ we have that $L = 0$ and $ \Sigma K+K^*\Sigma = 2K^*\Sigma^{-1}$. Hence, 
        \begin{equation}\label{4.3}
            \quad \Sigma K = K^*(2\Sigma^{-1} - \Sigma).
        \end{equation}
        Then, $\Sigma K K^*(2\Sigma^{-1} - \Sigma) = K^*(2\Sigma^{-1} - \Sigma) \Sigma K$, i.e.,  $K(2I_r - \Sigma^{2}) = (2I_r - \Sigma^{2})K$, 
        which gives $K\Sigma^{2}=\Sigma^{2}K$. Applying Lemma~\ref{y2}, $K\Sigma = \Sigma K$ holds, which together with (\ref{4.3}) implies that $K^2 = 2\Sigma^{-2} - I_r$.
       Similarly as in the part ($a$), we can prove that $\Sigma = I_r$, and (\ref{4.3}) gives $K^2 = I_r$. By Lemma~\ref{y5}, we conclude $A \in \mathbb{C}_n^{3\text{-}OP}$.
\end{proof}

\noindent Based on Theorem~{\ref{za}}, it is natural to ask whether the equation
$\frac{1}{3}(A + A^{*} + A^{\dagger}) = A^{*}$ is equivalent to the condition \(A \in \mathbb{C}_{n}^{3\text{-}OP}\). The following example shows that this equivalence does not hold in general.
\begin{example} \label{range}
Let
\begin{eqnarray*}
A = 
\begin{pmatrix}
    1 & 0 & 0 & 0 \\
    0 & -1 & 0 & 0 \\
    0 & 0 & \dfrac{1}{\sqrt{3}}i & 0 \\
    0 & 0 & 0 & -\dfrac{1}{\sqrt{3}}i
\end{pmatrix}.
\end{eqnarray*}
Then,
\begin{eqnarray*}
A^* = 
\begin{pmatrix}
    1 & 0 & 0 & 0 \\
    0 & -1 & 0 & 0 \\
    0 & 0 & -\dfrac{1}{\sqrt{3}}i & 0 \\
    0 & 0 & 0 & \dfrac{1}{\sqrt{3}}i
\end{pmatrix}
\quad \text{and} \quad
A^{\dagger} = 
\begin{pmatrix}
    1 & 0 & 0 & 0 \\
    0 & -1 & 0 & 0 \\
    0 & 0 & -\sqrt{3}\,\vphantom{\dfrac{1}{2}}i & 0 \\
    0 & 0 & 0 & \sqrt{3}\,\vphantom{\dfrac{1}{2}}i
\end{pmatrix}.
\end{eqnarray*}
Thus,
\begin{eqnarray*}
A + A^{\dagger} = 
\begin{pmatrix}
    2 & 0 & 0 & 0 \\
    0 & -2 & 0 & 0 \\
    0 & 0 & -\dfrac{2}{\sqrt{3}}i & 0 \\
    0 & 0 & 0 & \dfrac{2}{\sqrt{3}}i
\end{pmatrix}
\quad \text{and} \quad
2A^* = 
\begin{pmatrix}
    2 & 0 & 0 & 0 \\
    0 & -2 & 0 & 0 \\
    0 & 0 & -\dfrac{2}{\sqrt{3}}i & 0 \\
    0 & 0 & 0 & \dfrac{2}{\sqrt{3}}i
\end{pmatrix}.
\end{eqnarray*}
Therefore, $A$ satisfies $A + A^{\dagger} = 2A^*$, i.e., $\frac{1}{3}(A + A^{*} + A^{\dagger}) = A^{*}$, but evidently $A^{*} \neq A$, so $A$ is not an orthogonal tripotent matrix.

\end{example}

\noindent  In the next theorem, we discuss the relation between
$\frac{1}{3}(A + A^{*} + A^{\dagger}) = A^{*}$ and \(A \in \mathbb{C}_{n}^{3\text{-}OP}\). 

\begin{theorem}\label{range} Let $A\in\mathbb{C}^{n\times n}$. The following are equivalent:

\noindent $(a)$  $A$ is an orthogonal tripotent matrix;

\noindent $(b)$ 	$\frac{1}{3}(A + A^{*}+A^{\dagger}) = A^{*}$, and $\pm\frac{1}{\sqrt{3}}i$ is not eigenvalues of $A$.
\end{theorem}
\begin{proof}  $(a)\Rightarrow(b):$ Evidently.

\noindent  $(b)\Rightarrow(a):$ Assume that $(b)$ holds. Then $A + A^{\dagger}=2A^*$ holds and if we multiply it from left and right  by $A$ we get 
    $A^2 + AA^{\dagger} = 2AA^*$ and  $A^2 + A^{\dagger}A = 2A^*A$. Thus
    \begin{equation}\label{4.5}
    2(AA^* - A^*A) = AA^{\dagger} - A^{\dagger}A.
    \end{equation}
    We recall that $\mathcal{R}(A^{\dagger}) = \mathcal{R}(A^*)$, so by $A + A^{\dagger}=2A^*$ we have the following
    \begin{eqnarray*}\mathcal{R} (A)=\mathcal{R} (2A^* - A^{\dagger}) \subseteq \mathcal{R} (A^*) + \mathcal{R} (A^{\dagger}) = \mathcal{R} (A^*). \end{eqnarray*}
  Now, $\mathcal{R} (A)\subseteq \mathcal{R} (A^*)$ and $r(A)=r(A^*)$, implying $\mathcal{R} (A) = \mathcal{R} (A^*)$. Thus  $AA^{\dagger} = A^{\dagger}A$.
    Therefore, from (\ref{4.5}), we obtain $AA^* = A^*A $. Thus $A$ is diagonalizable, so there exists $W \in \mathbb{C}_n^{U}$ such that
    \begin{eqnarray}\label{eA}
    A = W\,\mathrm{diag}(\lambda_1, \cdots, \lambda_r, 0, \cdots, 0)\,W^*,
    \end{eqnarray}
    where $\lambda_j \neq 0$ for $j = \overline{1,r}$ and $r = r(A)$.
    Consequently, $A + A^{\dagger}=2A^*$ gives  $\lambda_j+\lambda_j^{-1} = 2\overline{\lambda}_j$, i.e.,  $\lambda_j^2 + 1 = 2\overline{\lambda}_j{\lambda}_j$. Solving  the last equation, we obtain that  for all $j$, we have $\lambda_j\in\{-1,1, -\frac{1}{\sqrt{3}}i, \frac{1}{\sqrt{3}}i\}$. Since by our assumption $\pm\frac{1}{\sqrt{3}}i$ are not eigenvalues of $A$ we  conclude that  for all $j$, we have $\lambda_j\in\{-1, 1\}$. Now by $(\ref{eA})$ it is easy to check that  $A \in \mathbb{C}_{n}^{3\text{-}OP}$ by Theorem~\ref{zq}. 
\end{proof}

\noindent The following theorem establishes that the orthogonal tripotent matrices are equivalent to solutions of linear equations involving $ A $, $ A^* $, and $ A^\dagger $.

\begin{theorem}\label{range}  Let $ A \in \mathbb{C}^{n \times n} $. The following are equivalent: 

  \noindent $(a)$  $A$ is an orthogonal tripotent matrix;\   $(b)$ $A^{*}+AA^{\dagger}A^*=A + AA^{*}A$;\\
   \noindent    $(c)$ $A+AA^{\dagger}A^*=A^{\dagger}+AA^*A^{\dagger}$;\ \   $(d)$ $A+A^2A^{\dagger}=A^{\dagger}+A^{\dagger}AA^*$;\\
  \noindent   $(e)$ $A+A^2A^*=A^{\dagger}+A^{\dagger}AA^*$;\ \  \ \    $(f)$ $A+A^2A^*=A^*+A^*AA^*$;\\
    \noindent  $(g)$ $A^{\dagger}A+A^{\dagger}A^*=A^2+(A^*)^2$;\ \  \  $(h)$ $A^{\dagger}A+A^*A^{\dagger}=A^2+(A^*)^2$.
   
\end{theorem}
\begin{proof} By Theorem \ref{t3}, we have that $(a)$ implies each of the conditions $(b)$--$(h)$. We will now prove that $(b)$ implies $(a)$, noting that a similar argument shows that any of $(c)$--$(h)$ also implies $(a)$.

\noindent  $(b)\Rightarrow(a):$ From Lemma \ref{y1} and $A^{*} + AA^{\dagger}A^* = A + AA^{*}A$, we obtain 
$$\Sigma L+ \Sigma^3 L=0 \quad \text{and} \quad 2K^*\Sigma = \Sigma K+ \Sigma^3 K,$$ 
which immediately implies 
\begin{equation}\label{4.7} 
L = 0 \quad \text{and} \quad 2K^*\Sigma = (\Sigma + \Sigma^3)K. 
 \end{equation}
Consequently, $K^*\Sigma (\Sigma + \Sigma^3)K = (\Sigma + \Sigma^3)K K^*\Sigma$ holds, which implies  $(\Sigma^4 + \Sigma^2)K = K(\Sigma^4 + \Sigma^2).$ By Lemma \ref{y2}, we have $\Sigma K = K\Sigma$. Substituting this result into (\ref{4.7}) yields $2K^*\Sigma = K(\Sigma + \Sigma^3)$, from which it follows that  $\Sigma^2 + I_r = 2K^2 = 2(K^*)^2.$
Therefore, $(\Sigma^2 + I_r)^2 = 4K^2(K^*)^2 = 4I_r$, and factoring leads to  
$(\Sigma^2 - I_r)(\Sigma^2 + 3I_r) = 0.$ The invertibility of $\Sigma^{2}+3I_{r}$ implies $\Sigma = I_{r}$, which substituted into (\ref{4.7}) yields $K = K^{*}$. Thus, we have $A \in \mathbb{C}_n^{3\text{-}OP}$.
\end{proof}

\noindent We now characterize orthogonal tripotent matrices in terms of integer powers of $AA^*$ and $A^*A$.  
First, when $(A^\dagger)^\alpha = (A^\alpha)^\dagger$ holds for $\alpha \in \mathbb{N}$, Khatri \cite{Khatri1980} defined negative integer powers as 
$$A^{-\alpha} = (A^\dagger)^\alpha = (A^\alpha)^\dagger .$$
In the following, we describe the characterizations for the case where $s, t \in \mathbb{Z}$.


\begin{theorem}\label{zx}
    Let $A\in\mathbb{C}^{n\times n}$ and $s,t \in \mathbb{Z}$. The following are equivalent: 

  \noindent $(a)$  $A$ is an orthogonal tripotent matrix;\\
   \noindent  $(b)$ $A(AA^*)^{s}=A^{\dagger}(AA^*)^{t}$ with $s\neq t$ and $s - t + 1\neq0$;\\
   \noindent    $(c)$  $A^*(AA^*)^{s}=A(A^*A)^{t}$ with $s\neq t$ and $s + t + 1\neq0$;\\
     \noindent   $(d)$ $A(A^*A)^{s}=A^*(A^*A)^{t}$ with $s\neq t$ and $s - t + 1\neq0$;\\
  \noindent   $(e)$ $A(A^*A)^{s}=A^{\dagger}(A^*A)^{t}$ with $s\neq t$ and $s - t + 1\neq0$;\\   
    \noindent $(f)$ $A(AA^*)^{s}=A^*(AA^*)^{t}$ with $s\neq t$ and $s - t - 1\neq0$;\\
    \noindent  $(g)$ $A(A^*A)^{s}=A^{\dagger}(AA^*)^{t}$ with $s\neq -t$ and $s - t + 1\neq0$. 
 \end{theorem}
\begin{proof}  By Theorem \ref{t3}, we have that $(a)$ implies each of the conditions $(b)$--$(g)$. We will now prove that $(b)$ implies $(a)$, noting that a similar argument shows that any of $(c)$--$(g)$ also implies $(a)$.

\noindent  $(b)\Rightarrow(a):$  By Lemma \ref{y1}, it follows from $A(AA^{*})^{s}=A^{\dagger}(AA^{*})^{t}$ that
        \begin{equation}\label{4.15}
          L = 0 \quad \text{and} \quad \Sigma K = K^{*}\Sigma^{2(t - s)-1}, 
         \end{equation} 
         i.e.,  $\Sigma K K^{*}\Sigma^{2(t - s)-1}=K^{*}\Sigma^{2(t - s)-1} \Sigma K$,
          and consequently $ K\Sigma^{2(t - s)}=\Sigma^{2(t - s)}K.$ 
        For $s \neq t$, by Lemma \ref{y2} we get  $K\Sigma = \Sigma K$, which combined with (\ref{4.15}) gets
          \begin{eqnarray*}K\Sigma = K^{*}\Sigma^{2(t - s)-1}. \end{eqnarray*}
        Hence, $\Sigma^{2(s - t + 1)} = (K^*)^2 = K^2$, so $\Sigma^{4(s - t + 1)} = K^2(K^*)^2 = I_r$.
        For $s - t + 1 \neq 0$ implies $\Sigma = I_{r}$, and thus $K^2 = I_{r}$.
\end{proof}

\noindent As a corollary of Theorem~\ref{zx} in the cases where $s = 0$ or $t = 0$, we have the following results:

\begin{corollary}\label{lk} Let $A\in\mathbb{C}^{n\times n}$ and $s \in \mathbb{Z}\setminus\{0\}$. The following are equivalent: 

  \noindent $(a)$  $A$ is an orthogonal tripotent matrix;\ \   $(b)$$A = A^{\dagger}(AA^*)^{s}$ with $s\neq1$;\\
  \noindent  $(c)$ $A^{\dagger}=A(AA^*)^{s}$ with $s\neq - 1$;\ \   $(d)$ $A^* = A(A^*A)^{s}$ with $s\neq - 1$;\\
  \noindent  $(e)$ $A = A^*(AA^*)^{s}$ with $s\neq - 1$;\ \   $(f)$ $A = A^*(A^*A)^{s}$ with $s\neq 1$;\\
  \noindent  $(g)$ $A = A^{\dagger}(A^*A)^{s}$ with $s\neq 1$;\ \ 
 $(h)$  $A^{\dagger}=A(A^*A)^{s}$ with $s\neq - 1$;\\
  \noindent  $(i)$ $A^* = A(AA^*)^{s}$ with $s\neq1$.
   
\end{corollary}

\noindent Theorem \ref{zq} implies that $AA^* \in \mathbb{C}_{n}^{OP}$ when $A \in \mathbb{C}_{n}^{3\text{-}OP}$. 
We now characterize orthogonal tripotent matrices through this property. The proof will be ommited since the technique is the same as in the previous one.

\begin{theorem}\label{range} Let $ A \in \mathbb{C}^{n \times n} $. The following are equivalent: 

  \noindent $(a)$  $A$ is an orthogonal tripotent matrix;\ $(b)$ $AA^* \in \mathbb{C}_{n}^{OP}$ and $A = A^\dagger(AA^*)^s$;\\
   \noindent $(c)$ $AA^* \in \mathbb{C}_{n}^{OP}$ and $A^\dagger = A(AA^*)^s$;\ 
 $(d)$ $AA^* \in \mathbb{C}_{n}^{OP}$ and  $A^* = A(A^*A)^s$;\\
  \noindent $(e)$ $AA^* \in \mathbb{C}_{n}^{OP}$ and  $A = A^*(AA^*)^s$;\ 
  \noindent $(f)$ $AA^* \in \mathbb{C}_{n}^{OP}$ and  $A^* = A(AA^*)^s$; \\
  \noindent $(g)$ $AA^* \in \mathbb{C}_{n}^{OP}$ and  $A^{\dagger} = A(A^*A)^s$;\
  \noindent $(h)$ $AA^* \in \mathbb{C}_{n}^{OP}$ and  $A(AA^*)^{s}=A^{\dagger}(AA^*)^{t}$;\\
  \noindent $(i)$ $AA^* \in \mathbb{C}_{n}^{OP}$ and  $A^*(AA^*)^{s}=A(A^*A)^{t}$;\
  \noindent $(j)$ $AA^* \in \mathbb{C}_{n}^{OP}$ and  $A(A^*A)^{s}=A^*(A^*A)^{t}$;\\
  \noindent $(k)$ $AA^* \in \mathbb{C}_{n}^{OP}$ and  $A(AA^*)^{s}=A^*(AA^*)^{t}$;\
  \noindent $(l)$ $AA^* \in \mathbb{C}_{n}^{OP}$ and  $A(A^*A)^{s}=A^{\dagger}(AA^*)^{t}$.
\end{theorem}

\noindent We now characterize orthogonal tripotent matrices by combining linear equations involving the rank and trace of matrix. 

\begin{theorem}\label{t11} Let $ A \in \mathbb{C}^{n \times n} $. The following are equivalent: 

 \noindent $(a)$  $A$ is an orthogonal tripotent matrix;
   
 \noindent  $(b)$ $\operatorname{tr}(A^{*}A)+\operatorname{tr}((A^{*}A)^{\dagger}) = 2\operatorname{r}(A)$, $A^{*}A = A^{*}A^{\dagger}$; 
 
 \noindent  $(c)$ $\operatorname{tr}(A^{*}A)+\operatorname{tr}((A^{*}A)^{\dagger}) = 2\operatorname{r}(A)$, $A^{\dagger}A = A^{\dagger}A^{*}$; 
 
 \noindent  $(d)$ $\operatorname{tr}(AA^*)+\operatorname{tr}(A^*A)=2\operatorname{Re}(\operatorname{tr}(A^2)), A^*A=A^*A^{\dagger}$;
 
 \noindent  $(e)$ $\operatorname{tr}(AA^{\dagger})+\operatorname{tr}(A^3A^{\dagger}(A^*)^2)=2\operatorname{Re}(\operatorname{tr}(A^2)), AA^*=A^*A$.

\end{theorem}
\begin{proof}   By Theorem~\ref{zq},  it is clear that $(a)$ implies any of items $(b)$--$(e)$. Now, let us prove that $(b)$ implies $(a)$. 

\noindent  $(b)\Rightarrow(a):$  Note that $ \operatorname{tr}(A^{\dagger}A)=\operatorname{tr}((A^{\dagger}A)^*)=\operatorname{r}(A) $.
        Then $$\operatorname{tr}(A^{*}A) + \operatorname{tr}((A^{*}A)^{\dagger}) = 2\operatorname{r}(A)= \operatorname{tr}(A^{\dagger}A)+\operatorname{tr}((A^{\dagger}A)^*),$$ 
        i.e.,  $\operatorname{tr}[(A^* - A^{\dagger})(A^* - A^{\dagger})^*] = 0.$
        As $(A^* - A^{\dagger})(A^* - A^{\dagger})^*$ is semidefinite, we get  $A^* = A^{\dagger}$.
        Applying Lemma~\ref{y1}, we have $\Sigma = I_r$. 
        Moreover, by Lemma~\ref{y1} and $A^*A = A^*A^\dagger$, we obtain  
        \begin{equation*}
            K^*\Sigma^2L = 0, \quad L^*\Sigma^2L = 0 \quad \text{and} \quad K^*\Sigma^2K = K^*\Sigma K^*\Sigma^{-1},
        \end{equation*}
        from which left-multiplying the first two by $K$ and $L$, respectively, and summing yields $\Sigma^2L = 0$, i.e., $L = 0$. Consequently, the third equation reduces to $(\Sigma K)^2 = I_r$.
        Combining this with Lemma \ref{y5} we have 
        $$
        A^*A = A^*A^{\dagger} \Leftrightarrow  L = 0, (\Sigma K)^2 = I_r \Leftrightarrow  A = A^{\dagger}.
        $$ 
        Thus, $A = A^* = A^{\dagger}$, and it follows that $\mathbb{C}_n^{3\text{-}OP}$ by Theorem~\ref{t3}.
        
\noindent  The proof that any of $(c)$--$(e)$ implies $(a)$ is very similar to the previous one.     
\end{proof}

\begin{remark}\label{lpo}
    From \textup{Theorem~\ref{t11}($e$)}, we obtain:\\
     $(a)$  $A \in \mathbb{C}_{n}^{3\text{-}OP}$ if and only if $A=A^*$ and $\operatorname{r}(A) + \operatorname{tr}(A^4) = 2\operatorname{Re}(\operatorname{tr}(A^2))$;
     $(b)$  $A \in \mathbb{C}_{n}^{3\text{-}OP}$ if and only if $AA^*=A^*A$ and $\operatorname{r}(AA^{\dagger}) + \operatorname{tr}(A^2(A^*)^2) = 2\operatorname{Re}(\operatorname{tr}(A^2))$;
      $(c)$  For $k \in \mathbb{Z}^+$ with $k \geq 2$, $A^k=A$ if and only if $\operatorname{r}(A) + \operatorname{tr}(A^kA^{\dagger}(A^*)^{k-1}) = 2\operatorname{Re}(\operatorname{tr}(A^{k-1}))$.

\end{remark}


\begin{theorem}\label{range}  Let $ A \in \mathbb{C}^{n \times n} $. The following are equivalent: 

\noindent $(a)$  $A$ is an orthogonal tripotent matrix;\\
\noindent $(b)$ $A$ \mbox{is normal}, $AA^* \in \mathbb{C}_n^{OP}$ and $\operatorname{r}(A)+\operatorname{tr}(AA^{*}) = 2\operatorname{tr}(A^2)$;\\
\noindent $(c)$ $A$ \mbox{is normal}, $AA^* \in \mathbb{C}_n^{OP}$ and $\operatorname{r}(A)+\operatorname{tr}(AA^{\dagger}) = 2\operatorname{tr}(A^*A^{\dagger})$;\\
\noindent $(d)$ $A$ \mbox{is normal}, $AA^* \in \mathbb{C}_n^{OP}$ and $\operatorname{r}(A)+\operatorname{tr}(AA^{*}) = 2\operatorname{Re}(\operatorname{tr}((A^*)^{2}))$;\\
\noindent $(e)$ $A$ \mbox{is normal}, $AA^* \in \mathbb{C}_n^{OP}$ and $\operatorname{r}(A)+\operatorname{tr}(AA^*) = \operatorname{Re}(\operatorname{tr}(AA^{\dagger}) + \operatorname{tr}(A^{\dagger})^{2})$.
\end{theorem}

\begin{proof} By Theorem~\ref{zq}, evidently the condition $(a)$ implies each of conditions $(b)$--$(e)$.

\noindent $(b)\Rightarrow(a):$    If $A$ is normal,then $A$ is unitarily similar to a diagonal matrix, i.e.,  
    \begin{equation}\label{4.20}
    A = U\mathrm{diag}(\lambda_1, \cdots, \lambda_r, 0, \cdots, 0)U^*, 
    \end{equation}
    where $U \in \mathbb{C}_n^{U}$, $r=\operatorname{r}(A)$ and $\lambda_j$ for $j=\overline{1,r}$ are the non-zero eigenvalues of $A$. Since  $AA^*$ is idempotent, then for each $j$, we have $|\lambda_j|^4 = |\lambda_j|^2$, so $|\lambda_j| = 1$, $j=\overline{ 1,r}$. 
    From the condition $\operatorname{r}(A)+\operatorname{tr}(AA^{*}) = 2\operatorname{tr}(A^2)$, we obtain $r + \displaystyle\sum_{j=1}^r \lambda_j \overline{\lambda_j} = 2 \sum_{j=1}^r \lambda_j^2$, 
i.e.,  $r = \displaystyle\sum_{j=1}^r \lambda_j^2 $. Let $\lambda_j=a_j+ib_j$. Evidently, $|a_j| \leq 1 $, and $r = \displaystyle\sum_{j=1}^r \left(a_j^2 - b_j^2 + 2a_j b_j i\right)$. Thus $\displaystyle\sum_{j=1}^{r} 2a_jb_j i = 0$ and $r = \displaystyle\sum_{j=1}^{r} (a_j^2 - b_j^2)$.
Now, we have that $r \geq \displaystyle\sum_{j=1}^{r} a_j^2 = r + \displaystyle\sum_{j=1}^{r} b_j^2 \geq r$, so we conclude that $b_j = 0$ for all $j$, i.e., $\lambda_j = \pm 1$ for each $j$. Therefore, by Theorem \ref{zq}, we conclude that $A \in \mathbb{C}_n^{3\text{-}OP}$.

\noindent Other implications can be proved similarly. 
\end{proof}

\noindent  Motivated by (\ref{1.1}), in the next theorem, we separately consider the cases $A = A^*$ and $A = A^3$, and establish equivalent conditions for each equality.

\begin{theorem}\label{range}
     Let $A \in \mathbb{C}^{n \times n}$. Furthermore, let $(a)$--$(g')$ correspond to the following conditions:\\
  $\begin{aligned}
(a)\ & \operatorname{r}(A) + \operatorname{r}(I_n - A) + \operatorname{r}(I_n + A) = 2n; 
&& (a')\  A^{\dagger}A = A^{\dagger}A^{*}; \\
(b)\ & \mathcal{R}(A)\oplus \mathcal{R}(I_n - A)\oplus \mathcal{R}(I_n + A)=\mathbb{C}^{n};
&& (b')\  A^{\dagger}A = A^{*}A^{\dagger}; \\
(c)\ & \mathcal{N}(A)\oplus \mathcal{N}(I_n - A)\oplus \mathcal{N}(I_n + A)=\mathbb{C}^{n};
&& (c')\  A^{*}=A^2A^{\dagger}; \\
(d)\ & A^{k + 2}=A^k \text{ and } A \text{ is diagonalizable};
&& (d')\  A=AA^{\dagger}A^{*}; \\
(e)\ & A^{k + 2}=A^k \text{ and } A \in \mathbb{C}_n^{EP};
&& (e')\  A=A^{*}AA^{\dagger}; \\
(f)\ & A^{\dagger}(A^2)^*=A^{\dagger};
&& (f')\  A=A^{\dagger}A^{*}A; \\
(g)\ & A^{\dagger}=A^{\dagger}A^3A^{\dagger};
&& (g')\  A=A^{*}A^{\dagger}A.
\end{aligned}$\\
\noindent Then the following statements are equivalent:\\
\smallskip
\noindent $(i)$ $A \in \mathbb{C}_{n}^{3\text{-}OP}$;\\
\noindent $(ii)$ Any one of conditions $(a)$--$(g)$ holds along with any one of conditions $(a')$--$(g')$.
\end{theorem}

\begin{proof}
$(a):$ Note a known fact that
$$
\operatorname{r}\big(f(A)\big) + \operatorname{r}\big(g(A)\big) = n + \operatorname{r}\big(f(A)g(A)\big),
$$
where $f(x)$ and $g(x)$ are coprime complex coefficient polynomials. Applying this to $A$ and $I_n-A$, $A-A^2$ and $I_n+A$, we have
$$ 
\operatorname{r}(A) + \operatorname{r}(I_n - A) + \operatorname{r}(I_n + A) = n + \operatorname{r}(A - A^2) + \operatorname{r}(I_n + A) = 2n + \operatorname{r}(A - A^3).
$$
Therefore, $A^3 = A$ if and only if $\operatorname{r}(A) + \operatorname{r}(I_n - A) + \operatorname{r}(I_n + A) = 2n$.\\
$(a) \Leftrightarrow (b):$
Clearly, $\mathcal{R}(A) + \mathcal{R}(I_n - A) + \mathcal{R}(I_n + A) = \mathbb{C}^n$, which can be directly proved using part ($a$).\\
$(c):$
Since $A^3 = A$ if and only if $(A^*)^3 = A^*$, it follows that
\begin{align*}
 \mathcal{R} (A^*)\oplus \mathcal{R} (I_n - A^*)\oplus \mathcal{R} (I_n + A^*)=\mathbb{C}^{n} 
 \Leftrightarrow  \mathcal{N} (A)\oplus \mathcal{N} (I_n - A)\oplus \mathcal{N} (I_n + A)=\mathbb{C}^{n}.
\end{align*}  
$(d):$
 If $A$ is diagonalizable, there exists an invertible matrix $P$ such that
          \begin{eqnarray*}
          A = P \, \text{diag}(\lambda_1, \lambda_2, \ldots, \lambda_n) \, P^{-1},
          \end{eqnarray*}
          where $\lambda_i$ ($i = \overline{1,n}$) are the eigenvalues of $A$.
          From $A^{k+2} = A^k$, it follows that $\lambda_j^{k+2} = \lambda_j^k$ for each eigenvalue $\lambda_j$. Factoring yields $\lambda_j^k (\lambda_j^2 - 1) = 0$. Therefore, $\lambda_j \in \{0, 1, -1\}$.
          It is evident from Theorem~\ref{zq} that $A^3 = A$.\\
$(e):$ 
By Lemma~\ref{y5} and the condition $A \in \mathbb{C}_n^{\mathrm{EP}}$, we have $L = 0$. Combining this with the relation $A^{k+2} = A^k$ yields $(\Sigma K)^2 = I_r$, which consequently implies $A^3 = A$.\\
$(f):$
From $AA^{\dagger}(A^2)^* = AA^{\dagger}$, taking conjugate transposes yields $A^3A^{\dagger} = AA^{\dagger}$. Right-multiplying by $A$ and applying the identities $AA^{\dagger}A = A$ and $A^3A^{\dagger}A = A^3$ yields $A^3 = A$.\\
Since the condition $(g)$ is clearly equivalent to $A^3 = A$, it follows that one of all conditions $(a)$--$(g)$ are equivalent to $A^3 = A$.\\
\noindent Now let's consider the conditions in $(a')$--$(g')$. Only ($g'$) requires proof here since ($h'$) follows similarly and ($a'$)--($f'$) hold by direct computation.\\
We now prove ($g'$). From Lemma~\ref{y1} and $A = A^{\dagger}A^{*}A$, we obtain
\begin{align}
K^*\Sigma^{-1}K^*\Sigma^2K &= \Sigma K \quad \text{and} \quad K^*\Sigma^{-1}K^*\Sigma^2L = \Sigma L; \label{4.24} \\
L^*\Sigma^{-1}K^*\Sigma^2K &= 0 \quad \text{and} \quad L^*\Sigma^{-1}K^*\Sigma^2L = 0. \label{4.25}
\end{align}
By right-multiplying \eqref{4.24} by $K$ and $L$, respectively, summing the results, and performing the same operations on \eqref{4.25}, we obtain
\begin{equation}\
K^*\Sigma^{-1}K^*\Sigma^2 = \Sigma \quad \text{and} \quad L^*\Sigma^{-1}K^*\Sigma^2 = 0.\label{4.26}
\end{equation}
Applying a similar argument to \eqref{4.26} yields $\Sigma^{-1}K^*\Sigma^2 = K$. Substituting into $K^*\Sigma^{-1}K^*\Sigma^2 = \Sigma$ gives $K^*K = I_r$, implying $L = 0$. Then from (\ref{4.24}) we obtain $K^*\Sigma = \Sigma K$.\\
    Hence, by Lemma~\ref{y1}, we have
    \begin{align*}
    &A^{\dagger}A = A^{\dagger}A^* \Leftrightarrow A^{\dagger}A = A^*A^{\dagger} \Leftrightarrow A^{*}=A^2A^{\dagger} \Leftrightarrow A=AA^{\dagger}A^{*} \\
    \Leftrightarrow &A=A^{\dagger}AA^{*} \Leftrightarrow A=A^{*}AA^{\dagger} \Leftrightarrow A=A^{\dagger}A^{*}A \Leftrightarrow A=A^{*}A^{\dagger}A \\
    \Leftrightarrow &L=0,  K^*\Sigma = \Sigma K \Leftrightarrow A^* = A. 
    \end{align*}
By (\ref{1.1}), the satisfaction of one condition from $(a)$--$(g)$ and one from $(a')$--$(g')$ implies $A \in \mathbb{C}_n^{3\text{-}OP}$. 
  \end{proof}  
\noindent Furthermore, we try to weaken the constraints on parameters $s$ and $t$ in Theorem~\ref{zx}. This is done by building a matrix rank equality condition.

\begin{theorem}\label{range} Let $A \in \mathbb{C}^{n \times n}$ and $s,t \in \mathbb{Z}$. Furthermore, let $(a)$--$(g')$ correspond to the following conditions:\\
\smallskip
    \noindent $(a)$ $\operatorname{r}(I_{n}-AA^{*}) = n - \operatorname{r}(A)$;\\
    \noindent $(b)$ $\operatorname{r}(I_{n}-A^{*}A) = n - \operatorname{r}(A)$;\\
    \noindent $(c)$ $\operatorname{r}(I_{n}-A^{\dagger}A^{*}) = n - \operatorname{r}(A)$;\\
    \noindent $(d)$ $\operatorname{r}(I_{n}-A^{*}A^{\dagger}) = n - \operatorname{r}(A)$;\\
    \noindent $(a')$ $A = A^*(A^*A)^s$ with $s \neq 0$;\\
    \noindent $(b')$ $A^*(AA^*)^{s}=A(A^*A)^{t}$ with $s \neq t$;\\
    \noindent $(c')$ $A(A^*A)^{s}=A^*(A^*A)^{t}$ with $s \neq t$;\\
    \noindent $(d')$ $A(AA^*)^{s}=A^*(AA^*)^{t}$ with $s \neq t$;\\
    \noindent $(e')$ $A(A^*A)^{s}=A^{\dagger}(AA^*)^{t}$ with $s \neq t$;\\
    \noindent $(f')$ $A(AA^*)^{s}=A^{\dagger}(AA^*)^{t}$ with $s \neq t-1$;\\
    \noindent $(g')$ $A(A^*A)^{s}=A^{\dagger}(A^*A)^{t}$ with $s \neq t-1$.\\
\smallskip
\noindent Then the following statements are equivalent:\\
\smallskip
\noindent $(i)$ $A \in \mathbb{C}_{n}^{3\text{-}OP}$;\\
\noindent $(ii)$ Any one of conditions $(a)$--$(d)$ holds along with any one of conditions $(a')$--$(g')$.

\end{theorem}
\begin{proof}   
    From Lemma~\ref{y1}, we conclude that
    \begin{eqnarray*}
    I_n - AA^*=U\begin{pmatrix}I_r - \Sigma^2 & 0 \\ 0 & I_{n - r}\end{pmatrix}U^*.
    \end{eqnarray*}
    Given $\operatorname{r}(I_{n}-AA^{*}) = n - \operatorname{r}(A)$, it follows that $\operatorname{r}(I_r - \Sigma^2)=0$, which implies $\Sigma^2 = I_r$ and hence $\Sigma = I_r$.
    Note that $\operatorname{r}(I_{n}-AA^{*})=\operatorname{r}(I_{n}-A^{*}A)$, it follows from conditions $(a)$ and $(b)$ that
    \begin{eqnarray*}\operatorname{r}(I_{n}-AA^{*}) = n - \operatorname{r}(A) \Leftrightarrow  \operatorname{r}(I_{n}-A^{*}A) = n - \operatorname{r}(A) \Leftrightarrow  \Sigma = I_r.\end{eqnarray*}
    Similarly,
    \begin{eqnarray*}
    I_n - A^{\dagger}A^{*}=U\begin{pmatrix}I_r - K^*\Sigma^{-1}K^*\Sigma & 0 \\ 0 & I_{n - r}\end{pmatrix}U^*.
    \end{eqnarray*}
    Given $\operatorname{r}(I_{n}-A^{\dagger}A^{*}) = n - \operatorname{r}(A)$, we conclude that $\operatorname{r}(I_r - K^*\Sigma^{-1}K^*\Sigma)=0$, which implies $K^*\Sigma^{-1}K^*\Sigma = I_r$.
    By Theorem~\ref{zx}, all equivalent conditions in $(a')$--$(g')$ require $L = 0$, which when combined with $K^*\Sigma^{-1}K^*\Sigma = I_r$ yields $K^*\Sigma = \Sigma K$. This equality implies $\Sigma^2 K = K \Sigma^2$ through $K^*\Sigma \Sigma K = \Sigma K K^*\Sigma$, and consequently Lemma~\ref{y2} directly establishes $K\Sigma = \Sigma K$.
    Since $\operatorname{r}(I_{n}-A^{\dagger}A^{*}) = \operatorname{r}(I_{n}-A^{*}A^{\dagger})$, the conditions $(c)$ and $(d)$ imply for $L=0$ that 
    \begin{eqnarray*}\operatorname{r}(I_{n}-A^{\dagger}A^{*}) = n - \operatorname{r}(A) \Leftrightarrow  \operatorname{r}(I_{n}-A^{*}A^{\dagger}) = n - \operatorname{r}(A) \Leftrightarrow  K\Sigma = \Sigma K. \end{eqnarray*}
      Let us now consider conditions $(a')$--$(g')$. We present only the proof for condition ($b'$) satisfying any condition among $(a)$--$(d)$. The remaining conditions of $(a')$--$(g')$ follow similarly.\\
    Condition ($b'$) follows from Theorem~\ref{zx} and is equivalent to
    $$
    L = 0 \quad \text{and} \quad K^{*}\Sigma^{2s + 1} = \Sigma^{2t + 1}K.
    $$
    If condition ($a$) or ($b$) holds, $\Sigma = I_r$ trivially implies $K = K^*$. If condition ($c$) or ($d$), $K\Sigma = \Sigma K$ leads to $K^{*}\Sigma^{2s + 1} = K\Sigma^{2t + 1}$, from which $K^2 = (K^*)^2 = \Sigma^{2(s-t)}$ follows. Since $\Sigma^{4(s-t)} = I_r$ with $s \neq t$, we obtain $\Sigma = I_r$, thereby establishing $K = K^*$ again.\\
    Hence, by Lemma~{\ref{y5}}, we conclude that $A \in \mathbb{C}_{n}^{3\text{-}OP}$.\\
    For the converse implication, Theorem~\ref{t3} yields $A = A^* = A^{\dagger}$, which immediately implies that conditions $(a')$--$(g')$ hold, and moreover, $A^3 = A$. When $A = A^3$, $\operatorname{r}(I_n - A^2) = n - \operatorname{r}(A)$ always holds, and hence conditions $(a)$--$(d)$ are satisfied.
\end{proof}


\section{Conclusion}
In this paper, we explore the orthogonal tripotent matrices, which are an extension of the orthogonal idempotent (projection) matrices.
We use mean equations of $ A $, $ A^* $, and $ A^\dagger $, along with integer powers of $ AA^* $ and $ A^*A $, 
to reveal algebraic properties of orthogonal tripotent matrices. 
 Furthermore, by rank, trace, and Moore-Penrose inverse,  
 we establish necessary and sufficient conditions for orthogonal tripotent matrices. 
The following topics are proposed for further research:
\begin{itemize}
\item The study of orthogonal tripotent operators in Hilbert Spaces.
\item  The investigation of orthogonal idempotent and orthogonal tripotent matrices over dual numbers \cite{Cui,Qi}. 
\end{itemize}

\section*{Funding}
This work is supported by the National Natural Science Foundation of China (No.11961076).

\section*{Disclosure statement}

No potential conflict of interest was reported by the authors.


\end{document}